\documentclass{amsart}

\usepackage{amsfonts,amssymb,amsmath,amsthm}
\usepackage{url}
\usepackage{enumerate}

\newtheorem{thrm}{Theorem}[section]

\newtheorem{prop}[thrm]{Proposition}

\theoremstyle{definition}
\newtheorem{definition}[thrm]{Definition}
\newtheorem{remark}[thrm]{Remark}
\numberwithin{equation}{section}
\email{jvilla@correo.uaa.mx}
\thanks{Supported by the grand 118294 of CONACyT and grands PIM 08-2, 11-2 of UAA}

\begin{document}
\author{Jos\'{e} Villa}
\address{Departamento de Matem\'{a}ticas y F\'{\i}sica\\
Universidad Aut\'{o}noma de Aguascalientes\\
Av. Universidad 940, C.P. 20131, Aguascalientes, Ags.\\
Mexico}
\title[The Dirichlet Problem]{On the Dirichlet Problem}

\begin{abstract}
Using, as main tool, the convergence theorem for discrete martingales and
the mean value property of harmonic functions we solve, a particular case
of, Dirichlet problem.
\end{abstract}

\maketitle

\section{Introduction}

Let $(\mathbb{R}^{d},||\cdot ||)$ be the normed Euclidean space. If $%
A\subset \mathbb{R}^{d}$, we denote by $\overline{A}$ and $\partial A$ the
closure and frontier (or boundary) of $A$, respectively. Let us fix first
the object of study in this paper.

\textbf{The Dirichlet problem (DP): }Given a non-empty, bounded, and open
set $V\subset \mathbb{R}^{d}$ and a continuous function $f:\partial
V\rightarrow \mathbb{R}$, the Dirichlet problem consist in finding a unique
continuous function $f:\overline{V}\rightarrow \mathbb{R}$ such that 
\begin{equation*}
h(x)=f(x),\ \ \forall x\in \partial V,
\end{equation*}%
and having in $V$ continuous partial derivatives of second order which
satisfy Laplace's differential equation, i.e.,
\begin{equation*}
\Delta h(x)=\sum_{k=1}^{d}\frac{\partial ^{2}h(x)}{\partial x_{k}^{2}}=0,\ \
\forall x\in V.
\end{equation*}

The Dirichlet problem has a long history in pure and applied mathematics
(see \cite{Ka}, \cite{G-T}), and it is the basis for more elaborated
problems (see \cite{N-O}, \cite{G-T})). Such problems can be approached in
many different ways. In fact, they can be solved using techniques from
differential equations, Monte Carlo methods, stochastic differential
equations, potential theory, etc.

The Monte Carlo method, introduced by Metropolis and Ulam, is a proceeding
for solving physical problems by a method which essentially depends on a
statistical sampling technique. There are many studies of \textbf{(DP)}
using Monte Carlo techniques (see, for example, \cite{Mu}, \cite{N-O} and
the references there in). In some sense such methods were the motivation to
introduce a random sequence $(X_{n}^{v})_{n}$, $X_{1}^{v}=v\in V$, that
converges almost surely to a point $X_{\infty }^{v}$ belonging to $\partial V
$ (see Subsection 2.1). Using this convergence we are going to deduce that a
solution of \textbf{(DP)} has a specific representation, and of course this
expression implies uniqueness of \textbf{(DP)}. Actually, this interplay
between partial differential equations and stochastic methods was incited by
Kakutani \cite{Ka} who give a probabilistic representation of the solution
to \textbf{(DP)} in terms of certain functional of Brownian motion (see \cite%
{Ch-Wal} or \cite{K-S}).

In the one-dimensional case \textbf{(DP)} always has a solution, in fact it
is piecewise-linear. But for $d\geq 2$, Zaremba \cite{Za} observes that 
\textbf{(DP)} was not always solvable. Hence, the existence is the difficult
part of \textbf{(DP)}. However, restricting our attention to certain regions 
$V$ we get existence. More specifically, we introduce the Poincar\'{e}'s
regularity of $\partial V$ and we proof that the expression given in the
uniqueness is well defined and it is a solution to \textbf{(DP)}.

So, the present paper is a nice application of some elementary results of
martingale theory to a classical subject in mathematics (pure and applied),
as it is the Dirichlet problem.

In the next section we begin remembering a characterization of harmonic
functions, and we use this and the martingale convergence theorem to prove
uniqueness, in Subsection 2.1, and existence, in Subsection 2.2, of \textbf{%
(DP)}.

\section{Solving the Dirichlet problem}

Before we deal with the Dirichlet problem \textbf{(DP)} we introduce an
important class of differentiable functions which are close related with it.

As usual, by $d(x,A)$ we design the distance from the point $x\in \mathbb{R}%
^{d}$ to the set $A\subset \mathbb{R}^{d}$, to be precise%
\begin{equation*}
d(x,A)=\inf \{||x-y||:y\in A\}.
\end{equation*}%
Let $B_{r}(x)=\{y\in \mathbb{R}^{d}:||x-y||<r\}$ be the open ball of radius $%
r>0$ centered at $x\in \mathbb{R}^{d}$. We also define the sphere, $%
S_{r}(x)=\partial B_{r}(x).$

\begin{definition}
Let $A\subset \mathbb{R}^{d}$ be a non empty open set and $h:V\rightarrow 
\mathbb{R}$. We say that\newline
$(i)$ $h$ is \textit{harmonic }in $A$ if%
\begin{equation*}
\Delta h(x)=0,\ \ \forall x\in V\text{.}
\end{equation*}%
\newline
$(ii)$ $h$ has the \textit{mean value property} in $A$ if for each $x\in A$
and $r>0$, such that $\overline{B_{r}(x)}\subset A$, 
\begin{equation*}
h(x)=\frac{1}{\sigma (S_{r}(x))}\int_{S_{r}(x)}h(z)\sigma (dz),
\end{equation*}%
where $\sigma (dz)$ is the Lebesgue (area) measure on $S_{r}(x)$ and $\sigma
(S_{r}(x))=cr^{d-1}$, here $c>0$ is a constant.
\end{definition}

\begin{prop}
\label{caracterizaArmo}Let $A\subset \mathbb{R}^{d}$ be a non empty open
set. A function $h:A\rightarrow \mathbb{R}$ is harmonic in $A$ if and only
if it has the mean value property in $A$.
\end{prop}

\begin{proof}
See Theorem 2 in Section 4.3 of \cite{Ch-Wal}, or \cite{K-S}.
\end{proof}

In what follows we are going to consider $(V,f)$ as in \textbf{(DP)}. That
is, $V\subset \mathbb{R}^{d}$ is a set non-empty, open and bounded and $%
f:\partial V\rightarrow \mathbb{R}$ is a continuous function.

\subsection{Uniqueness}

Let $\vartheta _{1},\vartheta _{2},...$ be a sequence of random variables
(r.v.) defined on the same probability space $(\Omega ,\mathcal{F},P)$. Such
r.v. are independently and identically distributed with uniform distribution
on the unitary sphere, $S_{1}(0)\subset \mathbb{R}^{d}$. We denote by $%
E[\cdot ]$ the expectation with respect to $P$.

Let $v\in V$ and $0<r<1$ be arbitrary and fix. Define the random sequence%
\begin{eqnarray}
X_{r}^{v}(1) &=&v,  \label{defsupDir} \\
X_{r}^{v}(n+1) &=&X_{r}^{v}(n)+rd(X_{r}^{v}(n),\partial V)\vartheta _{n},\ \
n\geq 1.  \notag
\end{eqnarray}

The basic connection between harmonic functions and martingales is given by:

\begin{prop}
\label{ArmoMartinga}Let $h:\overline{V}\rightarrow \mathbb{R}$ be continuous
and harmonic in $V$, then the sequence $(h(X_{r}^{v}(n)))_{n}$ is a martingale with
respect $\mathcal{F}_{n}=\sigma (X_{r}^{v}(1),...,X_{r}^{v}(n))$, the
minimal $\sigma $-algebra such that $X_{r}^{v}(1),...,X_{r}^{v}(n)$ are
measurables.
\end{prop}

\begin{proof}
From the definition (\ref{defsupDir}) of $(X_{r}^{v}(n))_{n}$ it follows
immediately that $X_{r}^{v}(n)\in V$, for each $n\in \mathbb{N}$. Since $V$
is bounded, then $\overline{V}$ is compact, therefore $h$ continuous in $%
\overline{V}$ implies that $(h(X_{r}^{v}(n)))$ is an integrable sequence. On
the other hand, since $\vartheta _{n}$ is independent of $\mathcal{F}%
_{n}=\sigma (\vartheta _{1},...,\vartheta _{n-1})$ and $X_{r}^{v}(n)$ is $%
\mathcal{F}_{n}$ measurable we get (see Example 1.5 in Chapter 4 of \cite%
{Durrett})%
\begin{eqnarray*}
E[h(X_{r}^{v}(n+1))|\mathcal{F}_{n}] &=&\left. E\left[ h\left(
x+rd(x,\partial V)\vartheta _{n}\right) \right] \right\vert _{x=X_{r}^{v}(n)}
\\
&=&\left. \frac{1}{\sigma (S_{rd(x,\partial V)}(x))}\int_{S_{rd(x,\partial
V)}\left( x\right) }h(z)\sigma (dz)\right\vert _{x=X_{r}^{v}(n)} \\
&=&\left. h(x)\right\vert _{x=X_{r}^{v}(n)}=h(X_{r}^{v}(n)).
\end{eqnarray*}%
Observe that we have used Proposition \ref{caracterizaArmo} in the third
equality.
\end{proof}

In particular, for each $j\in \{1,...,d\}$ consider the function $h_{j}:%
\overline{V}\rightarrow \mathbb{R}$ defined by $h_{j}(x)=x_{j}$, where $%
x=(x_{1},...,x_{d})$. By the harmonicity of $h_{j}$ the preceding result
implies that $(h_{j}(X_{r}^{v}(n)))_{n}$ is a bounded martingale, then the
martingale convergence theorem (see Theorem 2.10 in Chapter 4 of \cite%
{Durrett}) yields that $\lim_{n\rightarrow \infty
}h_{j}(X_{r}^{v}(n)):=X_{r}^{v,j}(\infty )$, a.s. In this way, 
\begin{equation}
\lim_{n\rightarrow \infty }X_{r}^{v}(n)=X_{r}^{v}(\infty
):=(X_{r}^{v,1}(\infty ),...,X_{r}^{v,d}(\infty ))\text{, a.s.}
\label{defxn}
\end{equation}

\begin{prop}
\label{Lfro}Under the preceding notation, we have $X_{r}^{v}(\infty )\in
\partial V$, a.s.
\end{prop}

\begin{proof}
Suppose the contrary, this means that there exits a measurable set $\Omega
^{\prime }\subset \Omega $ with positive probability such that $%
\lim_{n\rightarrow \infty }X_{r}^{v}(n,\omega )=X_{r}^{v}(\infty ,\omega
)\notin \partial V$, for each $\omega \in \Omega ^{\prime }$. Because $%
\partial V$ is a closed set we have, $d(X_{r}^{v}(\infty ,\omega ),\partial
V)>0$. Then, there exists $n_{0}\in \mathbb{N}$ for which%
\begin{equation}
||X_{r}^{v}(n,\omega )-X_{r}^{v}(\infty ,\omega )||<\frac{r}{4}%
d(X_{r}^{v}(\infty ,\omega ),\partial V)\text{, \ }\forall n\geq n_{0}.
\label{conlimsuc}
\end{equation}%
From the triangle inequality we obtain%
\begin{equation}
||X_{r}^{v}(n_{0}+1,\omega )-X_{r}^{v}(n_{0},\omega )||<\frac{r}{4}%
d(X_{r}^{v}(\infty ,\omega ),\partial V).  \label{conpoarriba}
\end{equation}%
On the other hand, (\ref{defsupDir}) implies%
\begin{equation}
||X_{r}^{v}(n_{0},\omega )-X_{r}^{v}(n_{0}+1,\omega
)||=rd(X_{r}^{v}(n_{0},\omega ),\partial V).  \label{qqpas}
\end{equation}%
Using the inequality (\ref{conlimsuc}) we get%
\begin{eqnarray*}
d(X_{r}^{v}(\infty ,\omega ),\partial V) &\leq &||X_{r}^{v}(\infty ,\omega
)-X_{r}^{v}(n_{0},\omega )||+d(X_{r}^{v}(n_{0},\omega ),\partial V) \\
&\leq &\frac{r}{4}d(X_{r}^{v}(\infty ,\omega ),\partial
V)+d(X_{r}^{v}(n_{0},\omega ),\partial V),
\end{eqnarray*}%
then (\ref{qqpas}) yields%
\begin{eqnarray}
(1-\frac{r}{4})d(X_{r}^{v}(\infty ,\omega ),\partial V) &\leq
&d(X_{r}^{v}(n_{0},\omega ),\partial V)  \notag \\
&=&\frac{1}{r}||X_{r}^{v}(n_{0},\omega )-X_{r}^{v}(n_{0}+1,\omega )||.
\label{lfr}
\end{eqnarray}%
From (\ref{lfr}) and (\ref{conpoarriba}) we conclude that $1-\frac{r}{4}<%
\frac{1}{4}$, which is a contradiction to $0<r<1$.
\end{proof}

Now we are ready to deal with the uniqueness of the Dirichlet problem.

\begin{thrm}
There is at most one solution to $\mathbf{(DP})$.
\end{thrm}

\begin{proof}
Let $h:\overline{V}\rightarrow \mathbb{R}$ be a solution to $\mathbf{(}$%
\textbf{DP}$)$. By the continuity\ of $h$ in $\overline{V}$ we get%
\begin{equation*}
\lim_{n\rightarrow \infty }h(X_{r}^{v}(n))=f(X_{r}^{v}(\infty )),\text{ \
a.s.}
\end{equation*}%
Now, due to $(h(X_{r}^{v}(n)))_{n}$ is a martingale (see Proposition \ref%
{ArmoMartinga}) and dominated convergence theorem implies%
\begin{equation*}
h(v)=E[h(X_{r}^{v}(1))]=\lim_{n\rightarrow \infty
}E[h(X_{r}^{v}(n))]=E[f(X_{r}^{v}(\infty ))].
\end{equation*}%
Therefore, $h(v)=E[f(X_{r}^{v}(\infty ))]$, for each $v\in V$. This equality
implies the uniqueness of $\mathbf{(}$\textbf{DP}$)$.
\end{proof}

\subsection{Existence}

Let $0<r<1$ be fix. Define $h:\overline{V}\rightarrow \mathbb{R}$ as%
\begin{equation}
h(v)=\left\{ 
\begin{tabular}{ll}
$f(v),$ & $v\in \partial V,$ \\ 
$E[f(X_{r}^{v}(\infty ))],$ & $v\in V,$%
\end{tabular}%
\right.  \label{defh}
\end{equation}%
where $X_{r}^{v}(\infty )$ is given by (\ref{defxn}).

From the uniqueness argument we see that such function should be the
solution of \textbf{(DP)}, moreover it also suggest that $h$ does not depend
of $r$. We begin verifying that this is the case.

\begin{prop}
\label{hwelldef}The function $h$ given in $($\ref{defh}$)$ is well define.
\end{prop}

\begin{proof}
By the Tietze-Urysohn theorem (see (4.5.1) in \cite{Diu}) there exists a
continuous function $\bar{f}:V_{1}:=\{x\in \mathbb{R}^{d}:d(x,\overline{V}%
)<1\}\rightarrow \mathbb{R}$ such that $\bar{f}|_{\partial V}=f$. It is easy
to see that $\bar{f}\in L^{2}(V_{1})$, then there exists a sequence $%
(f_{\varepsilon })_{\varepsilon >0}$ of harmonic functions in $V_{1}$ such
that (see Proposition 21.2c in \cite{DiBe}) 
\begin{equation*}
\lim_{\varepsilon \downarrow 0}f_{\varepsilon }(x)=\bar{f}(x)\text{, \
uniformly in }\overline{V}\text{.}
\end{equation*}%
Let $r,s\in (0,1)$. The Proposition \ref{Lfro} and dominated convergence
theorem yields%
\begin{eqnarray*}
E[f(X_{r}^{v}(\infty ))] &=&E[\bar{f}(X_{r}^{v}(\infty ))] \\
&=&\lim_{\varepsilon \downarrow 0}E[f_{\varepsilon }(X_{r}^{v}(\infty ))] \\
&=&\lim_{\varepsilon \downarrow 0}\lim_{n\rightarrow \infty
}E[f_{\varepsilon }(X_{r}^{v}(n))] \\
&=&\lim_{\varepsilon \downarrow 0}\lim_{n\rightarrow \infty }f_{\varepsilon
}(v)=E[f(X_{s}^{v}(\infty ))].
\end{eqnarray*}%
In the last equality we have used Proposition \ref{ArmoMartinga}.
\end{proof}

As we will see the harmonicity of $h$ is the easy part of \textbf{(DP)}.

\begin{prop}
The function $h$ defined in (\ref{defh}) is harmonic in $V$.
\end{prop}

\begin{proof}
Let $v\in V$ and $0<s$ such that $\overline{B_{s}(v)}\subset V$. This
implies that 
\begin{equation}
r:=\frac{s}{d(v,\partial V)}<1.  \label{der}
\end{equation}%
Using the notation of Proposition \ref{hwelldef}, we have for $n=2,3,...$,
by Proposition \ref{ArmoMartinga}, that%
\begin{equation*}
E[\bar{f}(X_{r}^{v}(n))|X_{r}^{v}(2)]=\left. E[\bar{f}(X_{r}^{\vartheta
}(n-1))]\right\vert _{\vartheta =X_{r}^{v}(2)}.
\end{equation*}%
The dominated convergence theorem, for conditional expectations, guaranties
that 
\begin{equation*}
E[\bar{f}(X_{r}^{v}(\infty ))|X_{r}^{v}(2)]=\left. E[\bar{f}%
(X_{r}^{\vartheta }(\infty ))]\right\vert _{\vartheta =X_{r}^{v}(2)}.
\end{equation*}%
Therefore, by (\ref{der}),%
\begin{eqnarray*}
E[f(X_{r}^{v}(\infty ))] &=&E[E[f(X_{r}^{v}(\infty ))|X_{r}^{v}(2)]] \\
&=&E\left[ \left. E[f(X_{r}^{\vartheta }(\infty ))]\right\vert _{\vartheta
=X_{r}^{v}(2)}\right]  \\
&=&\frac{1}{\sigma (S_{s}\left( v\right) )}\int_{S_{s}\left( v\right)
}E[f(X_{r}^{z}(\infty ))]\sigma (dz).
\end{eqnarray*}%
The result follows from Proposition \ref{caracterizaArmo}.
\end{proof}

For $d\geq 2$, as we have already mention, it is needed to impose some
regularity condition on the frontier of $V$ in order to get the continuity
of $h$ in $\overline{V}$.

\begin{definition}
We say that $v\in \partial V$ is a regular point for $(V,f)$ if 
\begin{equation*}
\lim_{\substack{ x\rightarrow v  \\ x\in V}}E[f(X_{r}^{x}(\infty ))]=f(v).
\end{equation*}
\end{definition}

\begin{remark}
If each point of $\partial V$ is regular, then (\ref{defh}) is the solution
to $\mathbf{(DP)}$.
\end{remark}

As we have observed there exist $(V,f)$ such that $\mathbf{(DP)}$ does not
have a solution. Hence it is convenient to have a sufficient condition to
analyze the regularity of frontier points of $V$. This is the reason why we
introduce the following concept.

\begin{definition}
Let $v\in \partial V$. A continuous function $q_{v}:\overline{V}\rightarrow 
\mathbb{R}$ is called a barrier at $v$ if $q_{v}$ is harmonic in $V$, $%
q_{v}(v)=0$ and 
\begin{equation}
q_{v}(x)>0,\ \ \forall x\in \overline{V}\backslash \{v\}.  \label{limsp}
\end{equation}
\end{definition}

\begin{thrm}
Let $v\in \partial V$ be a point with a barrier $q_{v}$, then it is regular.
\end{thrm}

\begin{proof}
Let $M=\sup \{|f(x)|:x\in \partial V\}$. Let $\varepsilon >0$, then there
exists $\delta >0$ such that 
\begin{equation*}
x\in \partial V\text{, }||x-v||<\delta \Rightarrow |f(x)-f(v)|<\varepsilon .
\end{equation*}%
On the other hand, from (\ref{limsp})%
\begin{equation*}
K:=\inf \{q_{v}(z):||z-v||\geq \delta ,z\in \overline{V}\}>0,
\end{equation*}%
hence%
\begin{equation*}
K^{-1}q_{v}(z)\geq 1,\ \ \forall z\in \overline{V},\ ||z-v||\geq \delta .
\end{equation*}%
Therefore%
\begin{eqnarray*}
|f(x)-f(v)| &<&\varepsilon +2M \\
&\leq &\varepsilon +(2MK^{-1})q_{v}(x)\text{, \ }\forall x\in \partial V.
\end{eqnarray*}%
Let $(v_{k})$ be an arbitrary sequence in $V$ such that $\lim_{k\rightarrow
\infty }v_{k}=v$. Define $X_{r}^{v_{k}}(\infty )$ as we did in (\ref{defxn}%
). Proposition \ref{ArmoMartinga} implies,%
\begin{eqnarray*}
|E[f(v)]-E[f(X_{r}^{v_{n}}(\infty ))]| &\leq &E[|f(v)-f(X_{r}^{v_{n}}(\infty
))|] \\
&\leq &\varepsilon +(2MK^{-1})E[q_{v}(X_{r}^{v_{n}}(\infty ))] \\
&=&\varepsilon +(2MK^{-1})q_{v}(v_{n}).
\end{eqnarray*}%
From the continuity of $q_{v}$ we are done.
\end{proof}

As an application we get a classical condition for regularity of the points
in the frontier of $V$. A point $v\in \partial V$ satisfies Poincar\'{e}'s
condition if we have a ball $B_{s}(u)$ such that $\overline{V}\cap \overline{%
B_{s}(u)}=\{v\}$.

\begin{prop}
If $v\in \partial V$ satisfies Poincar\'{e}'s condition, then it is regular.
\end{prop}

\begin{proof}
Merely observe that $q_{v}:\overline{V}\rightarrow \mathbb{R}$, defined as,
\begin{equation*}
q_{v}(x)=\left\{ 
\begin{array}{ll}
\log \left( \frac{||x-u||}{s}\right) , & d=2, \\ 
s^{2-d}-||x-u||^{2-d}, & d\geq 3,%
\end{array}
\right. 
\end{equation*}
is a barrier at $v.$
\end{proof}

\end{document}